%% file: draft_Cayley_0403.tex
\documentclass[a4paper, 11pt]{amsart}
\usepackage{amsfonts,amssymb,amscd,amsmath,enumerate,verbatim,calc,graphicx}
\usepackage{multirow,bigdelim}
\usepackage{ascmac} 

\numberwithin{equation}{section}

\newtheorem{theorem}{Theorem}[section]
\newtheorem{prop}[theorem]{Proposition}

\newtheorem{lem}[theorem]{Lemma}
\newtheorem{conj}[theorem]{Conjecture}
\theoremstyle{remark}
\newtheorem{rem}[theorem]{Remark}

\theoremstyle{definition}

\newtheorem{problem}[theorem]{Problem}
\newtheorem{definition}[theorem]{Definition}

\newcommand{\Z}{\mathbb{Z}}
\newcommand{\R}{\mathbb{R}}
\newcommand{\F}{\mathbb{F}}
\newcommand{\C}{\mathbb{C}}
\newcommand{\Q}{\mathbb{Q}}
\newcommand{\PP}{\mathbb{P}}

\newcommand{\spp}{\mathrm{supp}}
\newcommand{\height}{\mathrm{ht}}
\newcommand{\dg}{\mathrm{deg}}
\newcommand{\vol}{\mathrm{Vol}}
\newcommand{\wt}{\mathrm{wt}}
\newcommand{\con}{\mathrm{conv}}

\newcommand{\xb}{{\bf x}}
\newcommand{\yb}{{\bf y}}

\newcommand{\eb}{{\bf e}}

\begin{document}
\title{Lattice simplices of maximal dimension \\
with a given degree}

\author{Akihiro Higashitani}
\thanks{
The author is partially supported by JSPS Grant-in-Aid for Young Scientists (B) $\sharp$17K14177.} 
\address{Akihiro Higashitani,
Department of Mathematics, Kyoto Sangyo University, 
Motoyama, Kamigamo, Kita-Ku, Kyoto, Japan, 603-8555}
\email{ahigashi@cc.kyoto-su.ac.jp}

\begin{abstract}It was proved by Nill that for any lattice simplex of dimension $d$ with degree $s$ 
which is not a lattice pyramid, the inequality $d+1 \leq 4s-1$ holds. 
In this paper, we give a complete characterization of lattice simplices satisfying the equality, i.e., 
the lattice simplices of dimension $(4s-2)$ with degree $s$ which are not lattice pyramids. 
It turns out that such simplices arise from binary simplex codes. 
As an application of this characterization, we show that such simplices are counterexamples for the conjecture known as ``Cayley conjecture''. 
Moreover, by modifying Nill's inequaitly slightly, we also see the sharper bound $d+1 \leq f(2s)$, 
where $f(M)=\sum_{n=0}^{\lfloor \log_2 M \rfloor} \lfloor M/2^n \rfloor$ for $M \in \Z_{\geq 0}$. 
We also observe that any lattice simplex attaining this sharper bound always comes from a binary code. 
\end{abstract}

\maketitle

\section{Introduction}\label{sec:in}

\subsection{Terminologies}
We say that a convex polytope $P \subset \R^d$ is a {\em lattice polytope} 
if all of its vertices belong to the standard lattice $\Z^d$. 
For two lattice polytopes $P, P' \subset \R^d$, we say that $P$ and $P'$ are {\em unimodularly equivalent} 
if there exist $f \in \mathrm{GL}_d(\Z)$ and ${\bf u} \in \Z^d$ such that $P'=f(P)+{\bf u}$. 
One of the main topics of the study on lattice polytopes is to give a classification of lattice polytopes up to unimodular equivalence.

For a lattice polytope $P \subset \R^d$ of dimension $d$, 
we consider the generating function $\sum_{n \geq 0}|n P \cap \Z^d|t^n$, called the {\em Ehrhart series}. 
It is well known that Ehrhart series becomes a rational function which is of the form 
$$\sum_{n \geq 0}|n P \cap \Z^d|t^n=\frac{h^*_P(t)}{(1-t)^{d+1}},$$
where $h^*_P(t)$ is a polynomial in $t$ with integer coefficients. 
The polynomial $h^*_P(t)$ appearing in the numerator of Ehrhart series is called the {\em $h^*$-polynomial} of $P$. 
Let $\dg(P)$ denote the degree of the $h^*$-polynomial of $P$. 
It is known that $$\dg(P)=d+1-\min\{m : mP^\circ \cap \Z^d \not= \emptyset\},$$ 
where $P^\circ$ denotes the interior of $P$. In particular, $\dg(P) \leq d$. 
Moreover, $h^*_P(1)/d!$ coincides with the volume of $P$, so using the notation $\vol(P)=h^*_P(1)$ is natural. 
We refer the reader to \cite{BR} for more detailed information on Ehrhart series or $h^*$-polynomials.

For a lattice polytope $P \subset \R^d$, let 
$$\mathrm{Pyr}(P)=\mathrm{conv}(\{(\alpha,0) \in \R^{d+1} : \alpha \in P\} \cup \{(0,\ldots,0,1)\}) \subset \R^{d+1}.$$
This new lattice polytope is said to be a {\em lattice pyramid} over $P$. 
It is not so hard to see that $h_P^*(t)=h_{\mathrm{Pyr}(P)}^*(t)$ (\cite[Theorem 2.4]{BR}). 
In particular, $\dg(P)=\dg(\mathrm{Pyr}(P))$.

\subsection{Main Results}
The following is one of the most interesting open problems in the theory of lattice polytopes:
\begin{problem}\label{mondai}
Given a nonnegative integer $s$, classify all lattice polytopes with degree $s$ 
which are not lattice pyramids over lower-dimensional ones up to unimodular equivalence. 
\end{problem}

Let $\eb_i$ denote the $i$th unit vector of $\R^d$ and ${\bf 0}$ its origin. 
Then any lattice polytope of dimension $d$ with degree 0 is unimodularly equivalent to 
the $d$-folded lattice pyramids over one lattice point $\{{\bf 0}\}$ ($0$-dimensional lattice polytope), 
i.e., $\con(\{{\bf 0}, \eb_1,\ldots,\eb_d\}) \subset \R^d$. 
Moreover, Batyrev and Nill completely solve Problem \ref{mondai} for the case $s=1$ (\cite{BN}). 

On the other hand, Nill proved the following: 
\begin{theorem}[{\cite[Theorem 7]{N08}}]\label{nill}
Let $c$ and $s$ be nonnegative integers. For a lattice polytope $P$ of dimension $d$ having at most $c+d+1$ vertices 
with $\dg(P)\leq s$, if $P$ is not a lattice pyramid over a lower-dimensional one, 
then $d+1 \leq c(2s+1) + 4s-1$ holds. In particular, when $P$ is a simplex (i.e. $c=0$), we have $d \leq 4s-2$. 
\end{theorem}

In this paper, we give a complete characterization of lattice simplices of maximal dimension for a given degree $s$, 
i.e., of dimension $(4s-2)$ with degree $s$, which are not lattice pyramids up to unimodular equivalence. 
\begin{theorem}[Main Result 1]\label{main1}
Given a positive integer $s$, let $\Delta$ be a lattice simplex of dimension $(4s-2)$ with degree $s$ which is not a lattice pyramid. 
Then $s=2^r$ for some $r \in \Z_{\geq 0}$. Moreover, $\Delta$ is uniquely determined by $r$ up to unimodular equivalence and 
$\Delta$ arises from the $(r+2)$-dimensional binary simplex code. 
More precisely, $\Delta$ is unimodularly equivalent to $\Delta(r+2)$. 
\end{theorem}
We will explain the binary simplex codes and clarify the lattice simplex $\Delta(r+2)$ arising from a binary simplex code in Section \ref{sec:code}.

Furthermore, by modifying Theorem \ref{nill}, we obtain the following (see Proposition \ref{Higashitani}): 
For a lattice simplex of dimension $d$ with degree $s$ which is not a lattice pyramid, we have the inequality $d+1 \leq f(2s)$, where 
$\displaystyle f(M):=\sum_{n = 0}^\infty \left\lfloor \frac{M}{2^n} \right\rfloor
=\sum_{n=0}^{\lfloor \log_2 M \rfloor} \left\lfloor \frac{M}{2^n} \right\rfloor$ 
for $M \in \Z_{\geq 0}$. Note that $f(M) \leq 2M-1$ holds in general and the bound $f(2s)$ is sharp. 
As the second main result of this paper, we will observe that lattice simplices of dimension $d$ with degree $s$ which are not lattice pyramids 
satisfying $d+1=f(2s)$ have the special property as follows. 
\begin{theorem}[Main Result 2]\label{main2}
Given a positive integer $s$, let $\Delta$ be a lattice simplex of dimension $(f(2s)-1)$ with degree $s$ which is not a lattice pyramid. 
Then $\Delta$ arises from a binary code. More precisely, we have $\Lambda_\Delta \subset \{0,1/2\}^{f(2s)}$.
\end{theorem}
We will explain what $\Lambda_\Delta \subset (\R/\Z)^{d+1}$ is in Section \ref{sec:notation}. 

\subsection{Cayley Conjecture}

Recently, {\em Cayley polytopes} or {\em Cayley decompositions} of lattice polytopes are well studied 
and play an essential role for the study of lattice polytopes (\cite{BH10, DN, DHNP, HNP, Ito}). 
Let us recall the notion of Cayley polytopes and Cayley decompositions. 
\begin{definition}
Let $P_1,\ldots,P_\ell \subset \R^{d+1-\ell}$ be lattice polytopes. 
The {\em Cayley sum} $P_1* \cdots *P_\ell$ of them is the convex hull of 
$(\{\eb_1\} \times P_1) \cup ( \{{\bf e}_2\} \times P_2) \cup \cdots \cup ( \{{\bf e}_\ell \times P_\ell\})$ 
in $\R^{d+1} \cong \R^\ell \times \R^{d+1-\ell}$, 
where ${\bf e}_1,\ldots,{\bf e}_\ell$ denote the unit vectors of $\R^\ell$. 
The lattice polytope of the form $P_1* \cdots *P_\ell$ is called a {\em Cayley polytope}. 
A {\em Cayley decomposition} of a lattice polytope $P \subset \R^d$ is a choice of unimodular equivalence classes of $P$ 
which is a Cayley sum of some lattice polytopes. 
\end{definition}
By definition, we see that a lattice polytope $P \subset \R^{d+1}$ is a Cayley sum of $\ell$ lattice polytopes 
if and only if $P$ is mapped onto a unimodualr simplex of dimension $\ell$ by a projection $\R^{d+1} \rightarrow \R^\ell$. 
Note that algebro-geometric interpretation of Cayley polytopes is also known by \cite{Ito}.

The following, known as {\em Cayley conjecture}, 
is one central problem which concerns a Cayley decomposition of lattice polytopes. 
\begin{conj}[{\cite[Conjecture 1.2]{DN}}]\label{yosou}
Let $P \subset \R^d$ be a lattice polytope of dimension $d$ with degree $s$. 
If $d > 2s$, then $P$ is a Cayley polytope of at least $(d+1-2s)$ lattice polytopes. 
\end{conj}

Several partial answers for this conjecture are known. 
For example, this is true for smooth polytopes (\cite{DN}) or Gorenstein polytopes (\cite{DHNP}). 
Moreover, by using the invariant $\mu(P) \in \Q_{\geq 0}$ satisfying $\mu(P) \leq d+1-s$, called {\em $\Q$-codegree}, 
it is proved in \cite{DHNP} that if $d > 2(d-\mu(P))+1$, then $P$ is a Cayley polytope of at least $(2\lceil\mu(P)\rceil-d)$ lattice polytopes. 
(Refer to \cite[Theorem 3.4]{DHNP} for the precise statement.) 
Moreover, a weak version of this conjecture is solved in \cite[Theorem 1.2]{HNP}, i.e., 
a bound for the number of Cayley summands is given by a quadratic of the degree.

In Section \ref{sec:counter}, we will claim that this conjecture does not hold in general. 
More precisely, we provide an example of a lattice simplex $\Delta$ of dimension $d$ with degree $s$ 
such that $d>2s$ but $\Delta$ is a Cayley polytope into {\em less than} $(d+1-2s)$ lattice polytopes. 
Actually, those counterexamples come from the lattice simplices appearing in Theorem \ref{main1}. 
Moreover, we will suggest a modification of Conjecture \ref{yosou}. 
We remark that the examples given in this paper leave the possibility open that $d > 2s$ still implies that 
$P$ is a Cayley polytope of at least two lattice polytopes.

\subsection{Organization}
This paper is organized as follows: 
In Section \ref{sec:notation}, we introduce a finite abelian group associated with a lattice simplex, 
which plays an essential role for the classification of unimodular equivalence classes for lattice simplices, and we prepare some lemmas. 
In Section \ref{sec:code}, we introduce a binary simplex code and the finite abelian group arising from it and discuss its properties. 
In Section \ref{sec:syoumei}, we prove Theorem \ref{main1} and Theorem \ref{main2}. 
In Section \ref{sec:counter}, we supply a counterexample for Conjecture \ref{yosou}. 

\subsection*{Acknowledgement}
The author would like to express a lot of thanks to Johannes Hofscheier and Benjamin Nill 
for many advices and helpful comments on the main results and Cayley conjecture. 
The author also would like to thank to Kenji Kashiwabara for many fruitful and intriguing discussions. 
The essential ideas of the proofs of the main results come from the discussions with him.

\bigskip

\input{preliminary}

\bigskip

\input{coding}

\bigskip

\input{proof_of_main}

\bigskip

\input{cayley_decomp}

\bigskip

%

\end{document}

%% file: preliminary.tex
\section{Finite abelian groups associated with lattice simplices}\label{sec:notation}

In this section, we introduce the finite abelian group associated with a lattice simplex 
and discuss some properties on a lattice simplex in terms of this group. 

Let $\Delta \subset \R^d$ be a lattice simplex of dimension $d$ with 
its vertices ${\bf v}_1,{\bf v}_2,\ldots,{\bf v}_{d+1} \in \Z^d$. 
We introduce 
$$\Lambda_\Delta = \left\{(x_1,x_2,\ldots,x_{d+1}) \in (\R/\Z)^{d+1} : \sum_{i=1}^{d+1} x_i{\bf v}_i \in \Z^d, \; \sum_{i=1}^{d+1}x_i \in \Z \right\}$$
equipped with its addition defined by $\xb+\yb=(x_1+y_1,\ldots,x_{d+1}+y_{d+1}) \in (\R/\Z)^{d+1}$ 
for $\xb=(x_1,\ldots,x_{d+1}) \in (\R/\Z)^{d+1}$ and $\yb=(y_1,\ldots,y_{d+1}) \in (\R/\Z)^{d+1}$. 
We can see that $\Lambda_\Delta$ is a finite abelian group.

Let $\mathcal{F}(d)$ denote the set of unimodular equivalence classes of lattice simplices of dimension $d$ with a fixed vertex order 
and let $\mathcal{A}(d)$ denote the set of finite abelian subgroups $\Lambda$ of $(\R/\Z)^{d+1}$ satisfying that 
the sum of all entries of each element in $\Lambda$ is an integer. 

Actually, the correspondence 
\begin{align*}
\mathcal{F}(d) \rightarrow \mathcal{A}(d); \;\;\; \Delta \mapsto \Lambda_\Delta
\end{align*}
provides a bijection (\cite[Theorem 2.3]{BH13}). 
In particular, a unimodular equivalence class of lattice simplices $\Delta$ is uniquely determined 
by the finite abelian group $\Lambda_\Delta$ up to permutation of coordinates. 

We can discuss $h^*_\Delta(t)$, $\dg(\Delta)$, $\vol(\Delta)$ and 
whether $\Delta$ is a lattice pyramid in terms of $\Lambda_\Delta$. 
We fix some notation: For $\xb=(x_1,\ldots,x_{d+1}) \in \Lambda_\Delta$, where each $x_i$ is taken with $0 \leq x_i <1$, 
\begin{itemize}
\item let $\height(\xb)=\sum_{i=1}^{d+1} x_i \in \Z_{\geq 0}$; 
\item let $\spp (\xb)= \{i \in [d+1] : x_i \not= 0\}$, where $[n]=\{1,\ldots,n\}$ for each $n \in \Z_{>0}$; 
\item let $\wt(\xb)=|\spp(\xb)|$. 
\end{itemize}
Then we have 
$$h^*_\Delta(t)=\sum_{\xb \in \Lambda_\Delta} t^{\height(\xb)}.$$
Consult, e.g., \cite[Corollary 3.11]{BR}. 
In particular, $$\dg(\Delta)=\max\{\height(\xb) : \xb \in \Lambda_\Delta\}\; \text{ and }\;\vol(\Delta)=|\Lambda_\Delta|.$$ 

In addition: 
\begin{lem}[{cf. \cite[Proposition 2.5]{BH13}}]\label{lem2}
Let $\Delta$ be a lattice simplex of dimension $d$. Then $\Delta$ is not a lattice pyramid if and only if 
there is $(x_1,\ldots,x_{d+1}) \in \Lambda_\Delta$ such that $x_i \not= 0$ for each $i \in [d+1]$. 
\end{lem}

We consider a finite abelian subgroup $\Lambda$ of $(\R/\Z)^{d+1}$ (not necessarily the sum of the entries is an integer), i.e., 
$\Lambda$ is more general than $\Lambda_\Delta$. 
We use the same notation $\spp(\xb), \wt(\xb)$ and $\height(\xb)$ for $\xb \in \Lambda$ as above. We also use the notation 
$$\spp(\Lambda):=\bigcup_{\xb \in \Lambda} \spp(\xb),  \;\; \wt(\Lambda):=\max\{\wt(\xb) : \xb \in \Lambda\}, \;\; 
\dg(\Lambda):=\max\{\height(\xb) : \xb \in \Lambda\}.$$
Notice that a lattice simplex $\Delta$ of dimension $d$ is not a lattice pyramid if and only if $|\spp(\Lambda_\Delta)|=d+1$ by Lemma \ref{lem2}. 

\begin{lem}[{cf. \cite[Lemma 11]{N08}}]\label{lem1}
Let $\Lambda$ be a finite abelian subgroup of $(\R/\Z)^e$ and let $\dg(\Lambda) = s$. 
Then $\wt(\xb) \leq 2s$ for each $\xb \in \Lambda$. 
\end{lem}
\begin{proof}
For each $\xb \in \Lambda$, let $-\xb$ denote the inverse of $\xb$. Then 
$$\wt(\xb)=\sum_{i \in \spp(\xb)} (x_i+(1-x_i)) = \height(\xb)+\height(-\xb) \leq 2s.$$
\end{proof}

For a positive integer $M$, let 
$$f(M):=\sum_{n = 0}^\infty \left\lfloor \frac{M}{2^n} \right\rfloor
=\sum_{n=0}^{\lfloor \log_2 M \rfloor} \left\lfloor \frac{M}{2^n} \right\rfloor.$$ 
\begin{prop}\label{Higashitani} 
Let $\Lambda$ be a finite abelian group of $(\R/\Z)^e$ with $e \geq 3$ 
and let $\wt(\Lambda) \leq M$ with $M \in \Z_{\geq 2}$. 
Assume that $\left| \spp(\Lambda) \right|=e$. Then the following assertions hold. 
\begin{itemize}
\item[(a)] One has $e \leq f(M) \leq 2M-1$ {\em (cf. \cite[Theorem 10]{N08})}.  
\item[(b)] If $e=2M-1$, then $M=2^r$ for some $r \in \Z_{\geq 1}$. 
\item[(c)] Let $e=f(M)$. Then there exist some elements $\xb_1,\ldots,\xb_{\lfloor \log_2M \rfloor +1} \in \Lambda$ 
such that $\displaystyle \bigcup_{i=1}^{\lfloor \log_2M \rfloor +1} \spp(\xb_i)=\spp(\Lambda)$ holds. 
Conversely, if there exist $\xb_1,\ldots,\xb_q \in \Lambda$ with $\bigcup_{i=1}^q \spp(\xb_i)=\spp(\Lambda)$, 
then $q \geq \lfloor \log_2M \rfloor +1$. 
\end{itemize}
\end{prop}
\begin{proof}
Although most parts of the statements can be obtained just by modifying the proof of \cite[Theorem 10]{N08} slightly, 
we give a precise proof for the completeness. 

\smallskip

\noindent
(a) First, we show the inequalities $e=\left| \bigcup_{\xb \in \Lambda} \spp(\xb) \right| \leq \sum_{n=0}^\infty \lfloor M/2^n \rfloor \leq 2M-1$. 
Let $\xb_1 \in \Lambda$ with $\wt(\xb_1)$ maximal. 
We choose some elements $\xb_1,\xb_2,\ldots \in \Lambda$ successively in a ``greedy'' manner such that $|I_j|$ is maximal, 
where $I_j=\spp(\xb_j) \setminus \bigsqcup_{i=1}^{j-1} I_i$ for $j \geq 2$ and $I_1=\spp(\xb_1)$. 
What we may prove is the inequality 
\begin{align}\label{simesu}
|I_j| \leq \left\lfloor \frac{|I_1|}{2^{j-1}} \right\rfloor \; \text{ for each } \; j \geq 2.
\end{align}
In fact, since $\bigcup_{\xb \in \Lambda} \spp(\xb) = \bigsqcup_{j \geq 1} I_j$, we obtain 
\begin{align}\label{ccc}
e=\left|\bigcup_{\xb \in \Lambda} \spp(\xb) \right| = \sum_{j \geq 1} |I_j| \leq \sum_{j \geq 1} \left\lfloor \frac{|I_1|}{2^{j-1}} \right\rfloor 
\leq \sum_{n \geq 0}\left\lfloor \frac{M}{2^n} \right\rfloor =f(M)\end{align}
from \eqref{simesu} and $|I_1|=\wt(\xb_1) \leq M$, and we also obtain 
\begin{align}\label{hutousiki}
f(M)=\sum_{n = 0}^\infty\left\lfloor \frac{M}{2^n} \right\rfloor = \sum_{n = 0}^{\lfloor \log_2 M \rfloor}\left\lfloor \frac{M}{2^n} \right\rfloor 
\leq \sum_{n = 0}^{\lfloor \log_2M \rfloor} \frac{M}{2^n} \leq M \cdot \frac{1-\frac{1}{2M}}{1-\frac{1}{2}}=2M-1.
\end{align}

We prove \eqref{simesu} by induction on the number $q$ of possible elements $\xb_1,\xb_2,\ldots,\xb_q$. 
The case $q=1$ is clear. Let $q>1$ and suppose that the assertion is true for $q-1$. Set $I_q'=I_{q-1} \cap \spp(\xb_q)$. Then we have 
$I_q \sqcup I_q' \subset \spp(\xb_q) \setminus \bigcup_{i=1}^{q-2} \spp(\xb_i)$. 
Since $\xb_{q-1}$ is chosen with $|I_{q-1}|$ maximal, we obtain \begin{align}\label{eq1}|I_q| + |I_q'| \leq |I_{q-1}|.\end{align}
On the other hand, by considering $\spp(\xb_{q-1}+\xb_q)$, we also have 
$(I_{q-1} \setminus I_q') \sqcup I_q \subset \spp(\xb_{q-1}+\xb_q) \setminus \bigcup_{i=1}^{q-2} \spp(\xb_i)$. 
By the maximality of $|I_{q-1}|$ again, we obtain 
\begin{align}\label{eq2} |I_{q-1}| - |I_q'| +|I_q| \leq |I_{q-1}|. \end{align}
Hence, by \eqref{eq1} and \eqref{eq2}, we see that $2|I_q| \leq |I_{q-1}|$. Thus, $|I_q| \leq \lfloor |I_{q-1}|/2 \rfloor$. 
By the inductive hypothesis, we conclude that 
$$|I_q| \leq \left\lfloor \frac{|I_{q-1}|}{2} \right\rfloor 
\leq \left\lfloor \frac{\left\lfloor \frac{|I_1|}{2^{q-2}} \right\rfloor}{2} \right\rfloor
=\left\lfloor \frac{|I_1|}{2^{q-1}} \right\rfloor.$$

\noindent
(b) Assume $e=2M-1$. Then one has $f(M)=2M-1$ by (a). Thus it directly follows from \eqref{hutousiki} that $M$ must be a power of $2$. 

\noindent
(c) Assume $e=f(M)$. Then the equalities of the inequalities in \eqref{ccc} are satisfied. 
In particular, one has $|I_j|=\lfloor M/2^{j-1} \rfloor$ for each $j$. 
This implies that the number $q$ of possible elements should be at least $\lfloor \log_2 M \rfloor +1$, 
otherwise $e=\sum_{j=1}^q |I_j|<f(M)$. 
\end{proof}

We also see the following observation: 
\begin{lem}\label{1/2}
Given an even number $M \geq 2$ and an integer $e \geq 3$, 
let $\Lambda$ be a finite abelian subgroup of $(\R/\Z)^e$ with $\wt(\Lambda)=M$. 
Let $\xb = (x_1,\ldots,x_e) \in \Lambda$ and $\xb' = (x_1', \ldots, x_e') \in \Lambda$ and 
assume that $$\wt(\xb)=\wt(\xb')=M \text{ and }|\spp(\xb) \cap \spp(\xb')|=M/2.$$ 
Then $x_i \in \{0,1/2\}$ and $x_i' \in \{0,1/2\}$ for each $1 \leq i \leq e$. 
\end{lem}
\begin{proof}
Let $\xb=(x_1,\ldots,x_e) \in \Lambda$ and $\xb'=(x_1',\ldots,x_e') \in \Lambda$ as in the statement. 
Note that $|\spp(\xb) \setminus \spp(\xb')|=|\spp(\xb') \setminus \spp(\xb)|=M/2$. 

First, consider $\xb+\xb'$. Since $\wt(\xb+\xb') \leq M$, it follows that 
\begin{align}\label{111}
x_i+x_i' = 1 \text{ for each }i \in \spp(\xb) \cap \spp(\xb').  
\end{align}

Next, consider $2\xb+\xb'$. Since $\wt(2\xb+\xb') \leq M$ and $2x_i+x_i'=x_i+1 \not\in \Z$ holds for each $i \in \spp(\xb) \cap \spp(\xb')$ by \eqref{111}, 
we have $2x_i=1$ for each $i \in \spp(\xb) \setminus \spp(\xb')$. Similarly, by considering $\xb+2\xb'$, we obtain 
\begin{align}\label{222}
x_i=1/2 \text{ for each }i \in \spp(\xb) \setminus \spp(\xb') \text{ and }x_i'=1/2 \text{ for each }i \in \spp(\xb') \setminus \spp(\xb). 
\end{align}

Next, consider $3\xb+\xb'$. By \eqref{111} and \eqref{222}, we see that 
$2x_i=1$ for each $i \in \spp(\xb) \cap \spp(\xb')$. Similarly, by considering $\xb+3\xb'$, we obtain 
\begin{align}\label{333}
x_i=x_i'=1/2 \text{ for each }i \in \spp(\xb) \cap \spp(\xb'). 
\end{align}

Therefore, \eqref{222} and \eqref{333} imply the desired conclusion. 
\end{proof}

%% file: coding.tex
\section{Binary simplex codes and the associated lattice simplices}\label{sec:code}

In this section, we introduce some elements of linear codes, especially, binary codes. 
We associate the finite abelian group of $(\R/\Z)^{d+1}$ (i.e., the lattice simplex of dimension $d$) with a binary simplex code $C \subset \F_2^{d+1}$. 
Binary simplex codes and the associated lattice simplices will play the important role in this paper. 

Linear subspaces of the vector space over a finite field are called a {\em linear code}. 
Let $\F_p=\{0,1,\ldots,p-1\}$ be a finite field of prime order $p$. 
We call a linear code {\em binary} if $p=2$. 
We set the map $g: \F_p^{d+1} \rightarrow (\R/\Z)^{d+1}$ defined by 
$g((a_1,a_2,\ldots,a_{d+1}))=(a_1/p,a_2/p, \ldots,a_{d+1}/p) \in (\R/\Z)^{d+1}$ for $(a_1,\ldots,a_{d+1}) \in \F_p^{d+1}$. 
Then, for a linear code $C \subset \F_p^{d+1}$, $g(C)$ can be regarded as a finite abelian subgroup of $(\R/\Z)^{d+1}$. 

We will often use the following notation in the remaining parts. 
\begin{itemize}
\item For a positive integer $r$, 
we consider all the points in $(r-1)$-dimensional projective space $\PP_2^{r-1}$ over $\F_2$. 
There are $(2^r-1)$ points in $\PP_2^{r-1}$. 
For each of those points in $\PP_2^{r-1}$, we associate the column vector and we denote by $A(r) \in \F_2^{r \times (2^r-1)}$ 
the $(r \times (2^r-1))$-matrix whose columns are those vectors. 
For example, $A(2)=\begin{pmatrix} 1 &1 &0 \\ 1 &0 &1 \end{pmatrix}$ 
and $A(3)=\begin{pmatrix} 1 &1 &1 &0 &1 &0 &0 \\ 1 &1 &0 &1 &0 &1 &0 \\ 1 &0 &1 &1 &0 &0 &1 \end{pmatrix}.$
\item Let $M(r) \in \{0,1/2\}^{r \times 2^r-1}$ denote the matrix $g(A(r))$. More precisely, 
$M(r)$ is the matrix obtained by replacing $1$ in $A(r)$ into $1/2$. 
For example, $M(2)=\begin{pmatrix} 1/2 &1/2 &0 \\ 1/2 &0 &1/2 \end{pmatrix}$ 
and $M(3)=\begin{pmatrix} 1/2 &1/2 &1/2 &0 &1/2 &0 &0 \\ 1/2 &1/2 &0 &1/2 &0 &1/2 &0 \\ 1/2 &0 &1/2 &1/2 &0 &0 &1/2 \end{pmatrix}.$
\item An {\em $r$-dimensional binary simplex code} is a binary linear code generated by the row vectors of $A(r)$. 
Note that the binary simplex code is equivalently a dual code of {\em Hamming code}. 
\end{itemize}

\begin{rem}Simplex codes play the central role in the paper \cite{BH13} 
for the classification of lattice simplices of dimension $d$ which are not lattice pyramids 
whose $h^*$-polynomials are of the form $1+at^s$, where $d \geq 4$, $a \geq 1$ and $1 < s < (d+1)/2$. 
\end{rem}
Throughout this paper, we will only treat a {\em binary} simplex code, while simplex codes are considered for any finite field. 

Given $r \in \Z_{\geq 0}$, let $B(r+2) \subset \{0,1/2\}^{4 \cdot 2^r-1}$ denote the abelian group 
arising from $(r+2)$-dimensional binary simplex code, 
i.e., $B(r+2)$ is the finite abelian group generated by the row vectors of $M(r+2)$. 
Since it is known that $\wt(\xb)=2^{r+1}$ for any $\xb \in B(r+2) \setminus \{{\bf 0}\}$, 
we see that the sum of all entries of each element in $B(r+2)$ is an integer. 
Thus, we can associate a lattice simplex from $B(r+2)$. Let $\Delta(r+2)$ denote a lattice simplex corresponding to $B(r+2)$. 
Namely, $\Lambda_{\Delta(r+2)}=B(r+2)$. 

We note some properties on $B(r+2)$, or equivalently, $\Delta(r+2)$. 
\begin{itemize}
\item $|B(r+2)|=\vol(\Delta(r+2))=2^{r+2}$. 
\item The $h^*$-polynomial of $\Delta(r+2)$ is of the form $h_{\Delta(r+2)}^*(t)=1+(2^{r+2}-1)t^s$ with $s=2^r$. 
In particular, $\deg ( \Delta(r+2))=2^r$. 
\item $\dim (\Delta(r+2) ) + 1 =2^{r+2} - 1 = 4 \cdot \deg (\Delta(r+2))-1$. 
\end{itemize}

Moreover, we observe the following lemma which we will use in the proof of Lemma \ref{key}: 
\begin{lem}\label{seisitu} 
Given $r \in \Z_{\geq 0}$, let $\Lambda$ be a finite abelian subgroup of $(\R/\Z)^{2^{r+2} -1}$ with $\wt(\Lambda)=2^{r+1}$. 
Let $\xb_1,\ldots,\xb_{r+2} \in \Lambda$ and let $A$ be the $(r+2) \times (2^{r+2}-1)$ matrix 
whose row vectors are those $\xb_1,\ldots,\xb_{r+2}$. Assume that the support matrix of $A$ is equal to $A(r+2)$. 
Then $A=M(r+2)$. 
\end{lem}
Here, the {\em support vector} of a (row or column) vector $(a_1,\ldots,a_e) \in (\R/\Z)^e$ means 
the $(0,1)$-vector $(\epsilon_1,\ldots,\epsilon_e) \in \{0,1\}^{r+2}$ 
such that $\epsilon_i=0$ if $a_i=0$ and $\epsilon_i=1$ if $a_i \neq 0$, 
and the {\em support matrix} of a matrix $A$ means the $(0,1)$-matrix 
whose row (or column) vectors consist of the support vectors of the row (or column) vectors of $A$. 
\begin{proof}[Proof of Lemma \ref{seisitu}]
Since the support matrix of $A$ is equal to $A(r+2)$, we see that 
$\wt(\xb_i)=\wt(\xb_j)=2^{r+1}$ and $|\spp(\xb_i) \cap \spp(\xb_j)|=2^r$ for each $1 \leq i \neq j \leq r+2$. 
Hence, $\xb_i \in \{0,1/2\}^{2^{r+2}-1}$ by Lemma \ref{1/2}. Therefore, one obtains $A=M(r+2)$, as required. 
\end{proof}

%% file: proof_of_main.tex
\section{Proofs of Theorem \ref{main1} and Theorem \ref{main2}}\label{sec:syoumei}

This section is devoted to giving our proofs of the main results, Theorem \ref{main1} and Theorem \ref{main2}. 

Throughout this section, let $\Delta$ be a lattice simplex of dimension $d$ with degree $s$ satisfying $d+1=f(2s)$ 
which is not a lattice pyramid. We will use the following notation. 
\begin{itemize}
\item Let $r=\lfloor \log_2 s \rfloor$. 
\item From $d+1=f(2s)$, Proposition \ref{Higashitani} (c) guarantees the existence of $\xb_1,\ldots,\xb_{r+2} \in \Lambda_\Delta$ 
with $\bigcup_{i=1}^{r+2}\spp(\xb_i)=\spp(\Lambda_\Delta)$. Let us fix such $\xb_1,\ldots,\xb_{r+2}$. 
(Remark that $M$ in Proposition \ref{Higashitani} is equal to $2s$ by Lemma \ref{lem1}, 
so we have $\lfloor \log_2 M \rfloor + 1=\lfloor \log_2 2s \rfloor +1 = r+2$.) 
\item Let $M \in (\R/\Z)^{(r+2) \times f(2s)}$ be the matrix whose row vectors are $\xb_1,\ldots,\xb_{r+2}$. 
\item Let $\Lambda \subset \Lambda_\Delta$ be the finite abelian subgroup generated by $\xb_1,\ldots,\xb_{r+2}$. 
\end{itemize}

\subsection{Key Lemma}
The following lemma will be the key for the proofs of our main results. 
\begin{lem}[Key Lemma]\label{key}
The matrix $M$ contains $M(r+2)$ as a submatrix. In particular, when $s=2^r$, $M$ coincides with $M(r+2)$. 
\end{lem}
\begin{proof}
For each $\emptyset \neq J \subset [r+2]$, let 
\begin{align}\label{aj}A_J=\bigcap_{j \in J} \spp(\xb_j) \setminus \bigcup_{\ell \in [r+2] \setminus J} \spp(\xb_\ell) \subset [d+1].\end{align} 
First, we claim that $A_J \neq \emptyset$ for any $J$. 
Let $J=\{i\}$ for some $1 \leq i \leq r+2$, i.e., a singleton. Suppose $A_{\{i\}} = \emptyset$, 
i.e., $\spp(\xb_i) \subset \bigcup_{\substack{1 \leq j \leq r+2 \\ j \neq i}} \spp(\xb_j)$. 
Then we have $\bigcup_{j=1}^{r+2} \spp(\xb_j)=\bigcup_{\substack{1 \leq j \leq r+2 \\ j \neq i}} \spp(\xb_j) = [d+1]$, 
a contradiction to Proposition \ref{Higashitani} (c). Hence, $A_{\{i\}} \neq \emptyset$. 
Let $J=\{j_1,\ldots,j_t\} \subset [r+2]$, where $1 \leq j_1 < \cdots < j_t \leq r+2$. We set 
\begin{align*}
\yb_i=\begin{cases}
\xb_{j_1}+\xb_i, \;\;&\text{ if }i \in \{j_2,\ldots,j_t\}, \\
\xb_i, &\text{ otherwise} 
\end{cases}
\end{align*}
for $i=1,\ldots,r+2$. Then we see that 
\begin{align*}
\spp(\Lambda_\Delta) &\supset \bigcup_{j=1}^{r+2}\spp(\yb_j)=\bigcup_{i=2}^t\spp(\xb_{j_1}+\xb_{j_i}) \cup \bigcup_{\ell \not\in \{j_2,\ldots,j_t\}} \spp(\xb_\ell) \\
&\supset \bigcup_{i=2}^t(\spp(\xb_{j_1}) \cup \spp(\xb_{j_i}) \setminus \spp(\xb_{j_1}) \cap \spp(\xb_{j_i})) \cup \bigcup_{\ell \not\in \{j_2,\ldots,j_t\}} \spp(\xb_\ell) \\
&=\bigcup_{i=1}^{r+2} \spp(\xb_i) = \spp(\Lambda_\Delta). 
\end{align*}
Thus, $\bigcup_{j=1}^{r+2}\spp(\yb_j)=\spp(\Lambda_\Delta)$. Hence, by the above discussion, we have $\spp(\yb_{j_1}) \setminus \bigcup_{i \neq j_1}\spp(\yb_i) \neq \emptyset$. 
Moreover, we also see that 
\begin{align*}\emptyset&\neq\spp(\yb_{j_1}) \setminus \bigcup_{i \neq j_1}\spp(\yb_i)
=\spp(\xb_{j_1}) \setminus \left(\bigcup_{i =2}^t\spp(\xb_{j_1}+\xb_{j_i}) \cup \bigcup_{\ell \not\in \{j_1,\ldots,j_t\}} \spp(\xb_\ell)\right)\\
&\subset\spp(\xb_{j_1}) \setminus \left(\bigcup_{i=2}^t(\spp(\xb_{j_1}) \cup \spp(\xb_{j_i}) \setminus \spp(\xb_{j_1}) \cap \spp(\xb_{j_i})) \cup \bigcup_{\ell \not\in \{j_1,\ldots,j_t\}} \spp(\xb_\ell) \right) \\
&=\bigcap_{j \in J}\spp(\xb_j) \setminus \bigcup_{\ell \not\in J}\spp(\xb_\ell)=A_J. 
\end{align*}

Hence, we conclude that $A_J \neq \emptyset$ for any non-empty $J \subset [r+2]$. 
This implies that all non-zero $(0,1)$-vectors of $\{0,1\}^{r+2}$ appear in $M$ as support vectors of its column vectors. 
In other words, $M$ contains a certain submatrix $M'$ whose support matrix is equal to $A(r+2)$. 
Therefore, by Lemma \ref{seisitu}, we obtain that $M'=M(r+2)$, as required. 
\end{proof}

\subsection{Proofs of Theorem \ref{main1} and Theorem \ref{main2}} 

We are now in the position to prove Theorem \ref{main1} and Theorem \ref{main2}. 
\begin{proof}[Proof of Theorem \ref{main1}]
Since $d+1=4s-1$, Proposition \ref{Higashitani} (b) says that $s=2^r$ for some $r \in \Z_{\geq 0}$. Moreover, 
Lemma \ref{key} says that $\Lambda$ is nothing but $B(r+2)$. 
Hence, it suffices to show that $\Lambda_\Delta \setminus \Lambda = \emptyset$. 

\noindent
{\bf (The first step)}: Let $A_J$ be the same as \eqref{aj} for $\emptyset \neq J \subset [r+2]$. 
Then we saw that $A_J \neq \emptyset$. Moreover, we can also see by definition that 
$A_J \cap A_{J'} = \emptyset$ for any $J,J' \subset [r+2]$ with $J \neq J'$. 
Hence, one has $\bigsqcup_{\emptyset \neq J \subset [r+2]} A_J=[2^{r+2}-1]=[d+1]$ and each $A_J$ should be a singleton. 

\noindent
{\bf (The second step)}: We will claim that if $\Lambda \subsetneq \Lambda_\Delta$, then 
\begin{align}\label{claim} 
\wt(\xb')=2s \text{ and } \xb' \in \{0,1/2\}^{d+1} \text{ for any }\xb' \in \Lambda_\Delta \setminus \Lambda. \end{align}
Let $\xb' \in \Lambda_\Delta \setminus \Lambda$ and fix $a \in \spp(\xb')$. Then, by the first step, 
there exists a unique $J = \{j_1,\ldots,j_t\} \subset [r+2]$ such that $A_J=\{a\}$. 

For $i=1,\ldots,r+2$, let 
\begin{align*}
\yb_i=\begin{cases}
\xb_{j_1}+\xb_i, \;\;&\text{ if }i \in \{j_2,\ldots,j_t\}, \\
\xb_i, &\text{ otherwise}. 
\end{cases}
\end{align*}
Then $\spp(\yb_{j_1}) \setminus \bigcup_{i \neq j_1}\spp(\yb_i)=A_J=\{a\}$. (See the proof of Lemma \ref{key}.) 
Let us consider the finite abelian subgroup generated by $\yb_1,\ldots,\yb_{j_1-1},\xb',\yb_{j_1+1},\ldots,\yb_{r+2}$, denoted by $\Lambda'$. 
Then it follows that $\bigcup_{i \neq j_1}\spp(\yb_i) \cup \spp(\xb')=\bigcup_{i=1}^{r+2}\spp(\xb_i)=\spp(\Lambda_\Delta)$. 
Hence, applying the same discussion as in the proof of Lemma \ref{key}, 
we obtain that $\Lambda'$ is equal to $B(r+2)$. In particular, one sees that $\wt(\xb')=2^{r+1}=2s$ and $\xb' \in \{0,1/2\}^{d+1}$. 

\noindent{\bf (The third step)}: 
Finally, suppose that there exists $\xb' \in \Lambda_\Delta$ with $\xb' \not\in \Lambda$. 
Then $\wt(\xb')=2s$ by \eqref{claim}. Moreover, for each $\xb \in \Lambda$, since $\xb+\xb' \in \Lambda_\Delta \setminus \Lambda$, 
we also have $\wt(\xb+\xb')=2s$ by \eqref{claim} again. From $\xb,\xb' \in \{0,1/2\}^{d+1}$, we see that $\wt(\xb+\xb')=\wt(\xb)+\wt(\xb')-2|\spp(\xb) \cap \spp(\xb')|$, 
and hence, we obtain $|\spp(\xb) \cap \spp(\xb')|=s$ for any $\xb \in \Lambda \setminus \{{\bf 0}\}$. 
Thus, $$\sum_{\xb \in \Lambda \setminus \{{\bf 0}\}} |\spp(\xb) \cap \spp(\xb')| = s (|B(r+2)|-1)=s(4s-1).$$
On the other hand, for each $a \in [4s-1]$, there exists a unique $J \subset [r+2]$ such that $A_J=\{a\}$ by the first step. 
For $\yb=\xb_{i_1}+\cdots+\xb_{i_\ell}$, where $1 \leq i_1<\cdots < i_\ell \leq r+2$, 
we see that $a \in \spp(\yb)$ if and only if $|\{i_1,\ldots,i_\ell\} \cap J|$ is odd. 
Since there are exactly $2s$ such $\yb$'s in $\Lambda \cong B(r+2)$, we obtain that $\Lambda$ is $2s$-fold covered by the non-zero elements in $\Lambda$. 
This implies that $\sum_{\xb \in \Lambda \setminus \{{\bf 0}\}} |\spp(\xb) \cap \spp(\xb')|$ should be divisible by $2s$. 
However, $s(4s-1)$ is not divisible by $2s$, a contradiction. 
\end{proof}

Actually, Theorem \ref{main1} is already known for the cases $r=0$ and $r=1$: 
\begin{rem}[{the case $r=0$: \cite{BN}}]\label{r=0}
It is shown in \cite[Theorem 2.5]{BN} that every lattice polytope with degree at most $1$ is 
either {\em Lawrence prism} or {\em an exceptional simplex} (see \cite{BN} for the detail). In particular, 
a lattice {\em simplex} with degree $1$ which is not a lattice pyramid 
is either a lattice segment $[0,a] \subset \R$ (with $a \geq 2$) or $\con(\{(0,0),(2,0),(0,2)\}) \subset \R^2$. 
From this result, we see that Theorem \ref{main1} is true for $r=0$ (i.e. $s=1$). 
In fact, $\con(\{(0,0),(2,0),(0,2)\})$ is unimodularly equivalent to $\Delta(2)$. 
\end{rem}
\begin{rem}[the case $r=1$: \cite{HH}]
In the upcoming paper \cite{HH}, the lattice simplices with degree $2$ which are not lattice pyramids will be completely characterized. 
It will be shown in that paper that the lattice simplex of dimension $6$ with degree $2$ which is not a lattice pyramid 
is unimodularly equivalent to $\Delta(3)$. 
\end{rem}

\begin{proof}[Proof of Theorem \ref{main2}]
We prove the assertion by induction on $r$. 
The case $r=0$ was already proved by Batyrev--Nill \cite{BN} (see Remark \ref{r=0}). Thus, we assume $r \geq 1$. 
Moreover, we notice that $f(2s)=2^{r+2}-1$ holds if and only if $s=2^r$, which is the case of Theorem \ref{main1}. 
Hence, we also assume $f(2s)=d+1>2^{r+2}-1$. 

First, we claim that $\Lambda \subset \{0,1/2\}^{d+1}$. Since $d+1=f(2s)$, 
we see from Lemma \ref{key} that $M$ contains $M(r+2)$ as a submatrix. 
Let $I \subset [d+1]$ be a set of the indices corresponding to the columns of such submatrix $M(r+2)$. 
Note that $|I|=2^{r+2}-1$ and $(x_i)_{i \in I} \in \{0,1/2\}^{2^{r+2}-1}$ 
and $\sum_{i \in I} x_i = 2^r$ for any $\xb=(x_1,\ldots,x_{d+1}) \in \Lambda \setminus \{{\bf 0}\}$. 
Let $\Lambda^c$ be the restriction of $\Lambda$ to the complement of $I$ (i.e. $[d+1] \setminus I \neq \emptyset$). 
Since $\deg(\Lambda) = s$ and $\deg(B(r+2))=2^r$, we have $\deg(\Lambda^c) = s-2^r$. Notice that $\lfloor \log_2(s-2^r)\rfloor \leq r-1$. 
Moreover, $|\spp(\Lambda^c)|=|\spp(\Lambda)| - (2^{r+2}-1)=f(2s)-(2^{r+2}-1)$. In addition, 
$$f(2(s-2^r))=\sum_{n=0}^\infty \left\lfloor \frac{2(s-2^r)}{2^n} \right\rfloor = f(2s)-(2^{r+2}-1)=|\spp(\Lambda^c)|.$$
Thus, by the inductive hypothesis, we obtain that $\Lambda^c \subset \{0,1/2\}^{f(2s)-(2^{r+2}-1)}$, 
i.e., one has $(x_i)_{i \in [d+1] \setminus I} \in \{0,1/2\}^{f(2s)-(2^{r+2}-1)}$ for any $(x_1,\ldots,x_{d+1}) \in \Lambda_\Delta$. 
Therefore, $\Lambda \subset \{0,1/2\}^{d+1}$. In particular, $\xb_i \in \{0,1/2\}^{d+1}$ for each $i$.

Our remaining task is to prove that $\Lambda_\Delta \setminus \Lambda \subset \{0,1/2\}^{d+1}$. 
Assume that $\Lambda_\Delta \setminus \Lambda \neq \emptyset$ and take some $\xb'=(x_1',\ldots,x_{d+1}') \in \Lambda_\Delta \setminus \Lambda$. 
For each $\emptyset \neq J \subset [r+2]$, let us consider 
\begin{align*}\widetilde{\xb_i}=\begin{cases}
\xb_i+\xb', \;\;&\text{if }i \in J, \\
\xb_i, \;&\text{if }i \not\in J. 
\end{cases}\end{align*}
Let $A_J$ be the same as \eqref{aj} 
and let $B'=\{ j \in \spp(\xb') : x_j'=1/2\}$. 
Since $\bigcup_{i=1}^{r+2} \spp(\xb_i)=[d+1]$, we see that $$\bigcup_{i=1}^{r+2}\spp(\widetilde{\xb_i}) = [d+1] \setminus A_J \cap B'.$$
Suppose that $\bigcup_{i=1}^{r+2} \spp(\widetilde{\xb_i}) \subsetneq [d+1]$ for any $\emptyset \neq J \subset [r+2]$. 
Namely, $A_J \cap B' \neq \emptyset$ for any $J$. Since $A_J \cap A_{J'} = \emptyset$ for $J \neq J'$, 
we see that $|B'| \geq 2^{r+2}-1$. This implies that $$\height(\xb') \geq \frac{1}{2}\cdot |B'| \geq \frac{2^{r+2}-1}{2}=2^{r+1}-\frac{1}{2}.$$ 
Hence $\height(\xb') \geq 2^{r+1}$. However, $\height(\xb') \leq s$ and $r= \lfloor \log_2 s \rfloor$, a contradiction. 
Therefore, there exists $J \subset [r+2]$ such that $\bigcup_{i=1}^{r+2} \spp(\widetilde{\xb_i}) = [d+1]$. 
Fix such $J$ and let $\widetilde{\Lambda}$ be the subgroup generated by $\widetilde{\xb_1},\ldots,\widetilde{\xb_{r+2}}$. 
By applying the same discussions as above for $\widetilde{\Lambda}$ instead of $\Lambda$, we obtain that $\widetilde{\xb_i} \in \{0,1/2\}^{d+1}$. 
In particular, we conclude that $\xb' \in \{0,1/2\}^{d+1}$, as required. 
\end{proof}

%% file: cayley_decomp.tex
\section{Counterexamples for Cayley Conjecture}\label{sec:counter}

In this section, we disprove Cayley conjecture (Conjecture \ref{yosou}). 
More concretely, we show that for any given $s \geq 2$, there exists a lattice simplex $\Delta$ of dimension $d$ with degree $s$ satisfying $d > 2s$ 
such that $\Delta$ cannot be decomposed into at least $(d+1-2s)$ lattice polytopes as Cayley polytopes (Proposition \ref{hanrei}). 
We also suggest a ``modified'' version of Cayley conjecture (Conjecture \ref{modify}) 
and show that any lattice simplex of dimension $(4s-2)$ with degree $s$ which is not a lattice pyramid 
satisfies this conjecture.

\subsection{Cayley decomposition for lattice simplices and $\Lambda_\Delta$}
First, we see the following proposition. 
This might be known as a folklore and the author knew this by the personal communication with Johannes Hofscheier. 
\begin{prop}\label{cayley}
A lattice simplex $\Delta \subset \R^d$ of dimension $d$ can be written as a Cayley sum of $\ell$ lattice simplices 
if and only if there are $\ell$ non-empty sets $A_1,\ldots,A_\ell$ with $\bigsqcup_{i=1}^\ell A_i = [d+1]$ such that 
$\sum_{j \in A_i}x_j$ is an integer for any $\xb=(x_1,\ldots,x_{d+1}) \in \Lambda_\Delta$ and $i=1,\ldots,\ell$. 
\end{prop}
\begin{proof}
Let ${\bf v}_1,{\bf v}_2,\ldots,{\bf v}_{d+1}$ be the vertices of $\Delta$. 

\noindent{\bf ``If'' part}: 
Recall from \cite[Section 2]{BH13} that for a finite abelian subgroup $\Lambda \in {\mathcal A}(d)$, 
the associated simplex $\Delta_\Lambda$ (i.e., $\Lambda_{\Delta_\Lambda}=\Lambda$) 
can be constructed as follows: Let $\pi:\R^{d+1} \rightarrow (\R/\Z)^{d+1}$ be the natural surjection and 
let $N=\pi^{-1}(\Lambda)$. Then $N$ is a lattice containing $\Z^{d+1}$. 
Let $\eb_1,\ldots,\eb_{d+1}$ be the standard basis for $\Z^{d+1}$ and let $\Delta_\Lambda=\con(\{\eb_1,\ldots,\eb_{d+1}\})$ 
be a lattice simplex with respect to the lattice $N$. Then $\Delta_\Lambda$ is a desired simplex. 

Let $A_1,\ldots,A_\ell$ be non-empty subsets of $[d+1]$ such that $\bigsqcup_{i=1}^\ell A_i = [d+1]$ 
satisfying that $\sum_{j \in A_i}x_j$ is an integer for any $\xb=(x_1,\ldots,x_{d+1}) \in \Lambda_\Delta$ and $i=1,\ldots,\ell$.  
We regard $\Delta$ as a lattice simplex $\con(\{\eb_1,\ldots,\eb_{d+1}\})$ with respect to $\pi^{-1}(\Lambda_\Delta)$. 
Fix $i_j \in A_j$ for each $1 \leq j \leq \ell$ and consider $\Delta'=\con(\{\eb_{i_1},\ldots,\eb_{i_\ell}\})$. 
Given $\xb'=(x_{i_1}',\ldots,x_{i_\ell}') \in \Lambda_{\Delta'}$, we define $\xb=(x_1,\ldots,x_{d+1}) \in (\R/\Z)^{d+1}$ 
by $x_i=0$ if $i \not\in \{i_1,\ldots,i_\ell\}$ and $x_{i_j}=x_{i_j}'$ for $j=1,\ldots,\ell$. 
Then $\xb \in \Lambda_\Delta$. Since $0 \leq x_i < 1$ and $\sum_{j \in A_i}x_j \in \Z$ for each $i$, 
we see that $x_{i_j}=0$ for each $j$, i.e., $\spp(\xb')=\emptyset$. 
Hence $\Lambda_{\Delta'}$ is trivial, i.e., $\Delta'$ is a unimodular simplex. 
Therefore, $\Delta$ can be written as a Cayley sum of $\ell$ lattice simplices. 

\noindent{\bf ``Only if'' part}: 
Assume that $\Delta$ can be written as a Cayley sum of $\ell$ lattice simplices. 
Then there is a linear lattice transformation $\phi:\Z^d \rightarrow \Z^{d+1}$ such that 
\begin{align*}
{\bf v}_1 \mapsto \eb_1+{\bf w}_1, \ldots, {\bf v}_{c_1} \mapsto \eb_1+{\bf w}_{c_1}, 
&{\bf v}_{c_1+1} \mapsto \eb_2+{\bf w}_{c_1+1},\ldots,{\bf v}_{c_2} \mapsto \eb_2+{\bf w}_{c_2}, \\
\ldots, &{\bf v}_{c_{\ell-1}+1} \mapsto \eb_\ell+{\bf w}_{c_{\ell-1}+1},\ldots,{\bf v}_{c_\ell} \mapsto \eb_\ell+{\bf w}_{c_\ell}
\end{align*}
for some $1 \leq c_1 < \cdots < c_\ell \leq d+1$ (after reordering ${\bf v}_1,\ldots,{\bf v}_{d+1}$ if necessary), 
where ${\bf w}_j$ is a lattice point in $\Z^{d+1}$ such that the first $\ell$ entries are all $0$. 

Take $\xb=(x_1,\ldots,x_{d+1}) \in \Lambda_\Delta$ arbitrarily. 
Then $\sum_{j=1}^{d+1} x_j\phi({\bf v}_j) \in \Z^{d+1}$ by $\sum_{j=1}^{d+1} x_j{\bf v}_j \in \Z^d$, 
and $\sum_{j=1}^{d+1}x_j \in \Z_{\geq 0}$. Thus, 
\begin{align*}
\sum_{j=1}^{d+1} x_j\phi({\bf v}_j) = \sum_{i=1}^\ell \left(\sum_{j=c_{i-1}+1}^{c_i} x_j({\bf e}_i+{\bf w}_j) \right) 
=\sum_{i=1}^\ell \left(\sum_{j=c_{i-1}+1}^{c_i} x_j \right) {\bf e}_i + \sum_{j=1}^{d+1} x_j{\bf w}_j \in \Z^{d+1}, 
\end{align*}
where $c_0=0$. Hence, we obtain $\sum_{j=c_{i-1}+1}^{c_i} x_j \in \Z_{\geq 0}$ for each $i=1,\ldots,\ell$, as required. 
\end{proof}

\subsection{$\Delta(r+2)$ and Cayley conjecture}

For a lattice polytope $P$, let 
$$C(P)=\max\{\ell : P \text{ can be written as a Cayley polytope of $\ell$ polytopes}\}.$$ 
Conjecture \ref{yosou} says that for a lattice polytope $P$ of dimension $d$ with degree $s$, 
the inequality $C(P) \geq d+1-2s$ might hold if $d >2s$.

\begin{lem}\label{deruta}
Let $\Delta(r+2)$ be a lattice simplex as in Section \ref{sec:code} for $r \in \Z_{\geq 0}$. Then 
$$C(\Delta(r+2)) \leq \left\lfloor\frac{2^{r+2}-1}{3}\right\rfloor.$$
\end{lem}
\begin{proof}
Consider the matrix $M(r+2)$. Then there is no pair of the same columns. 
Moreover, since $\xb=(x_1,\ldots,x_{d+1}) \in \{0,1/2\}^{d+1}$ for each $\xb \in B(r+2)$, where $d+1=2^{r+2}-1$, 
we see that $x_i+x_j \in \Z$ holds for $i\not=j$ if and only if $x_i=x_j$. 

Hence, if there exists $A$ such that $\sum_{i \in A}x_i \in \Z$ for any $\xb \in \Lambda_\Delta$, then $|A| \geq 3$. 
Therefore, by Proposition \ref{cayley}, we obtain 
$\displaystyle C(\Delta(r+2)) \leq \left\lfloor\frac{d+1}{3}\right\rfloor = \left\lfloor\frac{2^{r+2}-1}{3}\right\rfloor.$ 
\end{proof}

Let $d+1=\dim (\Delta(r+2))+1=2^{r+2}-1$ and $s=\deg (\Delta(r+2))=2^r$. It then follows from Lemma \ref{deruta} that 
\begin{align*}
d+1-2s-C(\Delta(r+2)) \geq 2^{r+2}-1-2^{r+1}-\left\lfloor\frac{2^{r+2}-1}{3}\right\rfloor \geq \frac{2^{r+1}-2}{3}. 
\end{align*}
Thus, $d+1-2s > C(\Delta(r+2))$ when $r \geq 1$. This implies that $\Delta(r+2)$ does not satisfy Conjecture \ref{yosou}. 
More generally, we see the following: 
\begin{prop}\label{hanrei}
For any $s \geq 2$, there exists a lattice simplex $\Delta$ of dimension $(f(2s)-1)$ with degree $s$ 
which is not a lattice pyramid such that $d+1-2s > C(\Delta)$. 
Namely, there exists a counterexample of Conjecture \ref{yosou} for any degree $s$ with $s \geq 2$. 
\end{prop}
\begin{proof}
Consider the binary expansion of $2s$, i.e., let $2s=2^{u_1}+\cdots+2^{u_p}$, where $u_1>\cdots>u_p \geq 1$. 
Then it is easy to see that $\displaystyle f(2s)=\sum_{i=1}^p (2^{u_i+1}-1)=4s-p$. 

Let $M(u_1,\ldots,u_p)$ be the $(\sum_{i=1}^p u_i+p) \times (\sum_{i=1}^p (2^{u_i+1}-1))$ matrix 
obtained by the direct sum of $M(u_1+1),\ldots,M(u_p+1)$, i.e., 
\begin{align*}
M(u_1,\ldots,u_p)=\begin{pmatrix}
M(u_1+1) &         &       &{\bf O} \\
         &M(u_2+1) &       & \\
         &         &\ddots & \\
{\bf O}  &         &       &M(u_p+1) 
\end{pmatrix}. 
\end{align*}
Let $\Delta$ be the lattice simplex associated with the finite abelian group generated by the row vectors of $M(u_1,\ldots,u_p)$. 
Then we see that $\dim \Delta+1$ is equal to the number of columns of $M(u_1,\ldots,u_p)$, which is equal to $f(2s)(=4s-p)$. 
Moreover, we can easily see that 
$\deg (\Delta)=s$. In addition, similar to Lemma \ref{deruta}, we have $\displaystyle C(\Delta) \leq \left\lfloor \sum_{i=1}^p \frac{2^{u_i+1}-1}{3} \right\rfloor.$ 
Thus, \begin{align*}
d+1-2s-C(\Delta) \geq 4s-p-2s-\left\lfloor \sum_{i=1}^p \frac{2^{u_i+1}-1}{3} \right\rfloor 
\geq 2s-p-\frac{4s-p}{3} =\frac{2s-2p}{3}.
\end{align*}
Therefore, $d+1-2s>C(\Delta)$ if $s \geq 2$, as desired. 
\end{proof}

\subsection{Modification of Cayley conjecture}

It turns out that Conjecture \ref{yosou} is not true in general. 
Instead, we suggest a modification of the conjecture as follows: 
\begin{conj}\label{modify}
Let $P \subset \R^d$ be a lattice polytope of dimension $d$ with degree $s$. 
If $d > (17s-4)/6$, then $P$ is a Cayley polytope of at least $(d+1-\lfloor (17s-4)/6 \rfloor)$ lattice polytopes. 
\end{conj}
This conjecture is weaker than the original Conjecture \ref{yosou}. 
We will verify that Conjecture \ref{modify} is true for lattice simplices of dimension $(4s-2)$ with degree $s$. 

\begin{lem}\label{sita}
Let $\Delta(r+2)$ be a lattice simplex as before for $r \in \Z_{\geq 0}$. Then 
$$C(\Delta(r+2)) \geq \frac{2^{r+2}-2^{r-1}-1}{3}$$ 
if $r$ is odd and 
$$C(\Delta(r+2))=\frac{2^{r+2}-1}{3}$$ 
if $r$ is even. 
\end{lem}
Note that $2^{r+2}-2^{r-1}-1 \equiv 0 \mod 3$ for odd $r \in \Z_{\geq 1}$ and 
$2^{r+2}-1 \equiv 0 \mod 3$ for even $r \in \Z_{\geq 0}$. 
\begin{proof}[Proof of Lemma \ref{sita}]
Thanks to Proposition \ref{cayley} and $M(r+2) \in \{0,1/2\}^{(r+2) \times (2^{r+2}-1)}$, 
we may consider a certain decomposition $\bigsqcup_{i=1}^\ell A_i = [d+1]$ 
such that the sum of all $j$th column vectors of $M(r+2)$ for $j \in A_i$ is an integer vector. 

First, let $r$ be odd. We will claim that there is a decomposition of the columns of $M(r+2)$, i.e., 
$[d+1]$ satisfying the certain property above into $(c_r+d_r)$ sets 
$A_1,\ldots,A_{c_r}, B_1,\ldots,B_{d_r}$, where $c_r$ and $d_r$ are defined by the recurrences 
\begin{align*}c_1=1, \; c_r=4c_{r-2}+1 \text{ for }r \geq 3, \quad d_1=1, \;d_r=4d_{r-2} \text{ for }r \geq 3,\end{align*}
respectively, and each $A_i$ consists of three indices and each $B_i$ consists of four indices. 
For example, we may set $A_1=\{2,3,4\}$ and $B_1=\{1,5,6,7\}$ for the case $r=1$ since 
both the sum of 2nd, 3rd and 4th columns and the sum of 1st, 5th, 6th and 7th columns are equal to $(1,1,1)^t$. Note that 
$M(3)=\begin{pmatrix} 1/2 &1/2 &1/2 &0 &1/2 &0 &0 \\ 1/2 &1/2 &0 &1/2 &0 &1/2 &0 \\ 1/2 &0 &1/2 &1/2 &0 &0 &1/2 \end{pmatrix}.$ 
As a consequence of this statement, we can see  by induction on $r$ that 
\begin{align*}
C(\Delta(r+2)) &\geq c_r+d_r \geq 4(c_{r-2}+d_{r-2})+1 = 4\cdot C(\Delta(r)) + 1 \\
&\geq 4\cdot\frac{2^r-2^{r-3}-1}{3}+1 = \frac{2^{r+2}-2^{r-1}-1}{3} 
\end{align*}
for any $r \geq 3$, as required. 
Here, we observe that $M(r+2)$ looks like as follows: 
\begin{align}\label{gyouretsu}\arraycolsep5pt
M(r+2)=\left(
\begin{array}{@{\,}cccc|cccc|cccc|ccc@{\,}}
1/2 &\cdots &1/2 &1/2 &1/2 &\cdots &1/2 &1/2 &0 &\cdots &0 &0 &0 &\cdots &0 \\
1/2 &\cdots &1/2 &1/2 &0 &\cdots &0 &0 &1/2 &\cdots &1/2 &1/2 &0 &\cdots &0 \\
\hline
  &       & &0     & &     & &0     &  &     & &0     &  &       & \\
  \multicolumn{3}{c}{\text{{\LARGE $M(r)$}}}&\vdots &\multicolumn{3}{c}{\text{{\LARGE $M(r)$}}} &\vdots&\multicolumn{3}{c}{\text{{\LARGE $M(r)$}}} &\vdots&\multicolumn{3}{c}{\text{{\LARGE $M(r)$}}} \\
  &       & &0     & &     & &0     &        & & &0     &        & & 
  \end{array}
\right)\end{align}
By the inductive hypothesis, we can decompose into 
$c_{r-2}$ sets of indices of the columns of $M(r)$ each of which consists of three indices and 
$d_{r-2}$ sets of indices of the columns of $M(r)$ each of which consists of four indices. 
Thus, by combining such decompositions of $M(r)$, we can easily obtain the desired decomposition of the columns of $M(r+2)$ 
into $(4c_r+1)$ sets each of which consists of three indices and $4d_r$ sets each of which consists of four indices. 

Let $r$ be even. In the similar way to the case $r$ is odd, we can prove by induction on $r$ that for any even $r \in \Z_{\geq 2}$, we have 
$$C(\Delta(r+2)) \geq 4 \cdot C(\Delta(r))+1 \geq 4 \cdot \frac{2^r-1}{3}+1 = \frac{2^{r+2}-1}{3}.$$
Note that the case $r=0$ is trivial. (Employing the same notation as above, we see that $c_0=1, c_r=4c_{r-2}+1$ for $r \geq 2$ and $d_r=0$ when $r$ is even.) 
It follows from this and Lemma \ref{deruta} that $C(\Delta(r+2))=(2^{r+2}-1)/3$ for the case $r$ is even. 
\end{proof}
Note that $17 \cdot 2^{r-1} -2 \equiv 0 \mod 3$ for odd $r \in \Z_{\geq 1}$ and 
$17 \cdot 2^{r-1} -2 \equiv 2 \mod 3$ for even $r \in \Z_{\geq 2}$. 
Thus, from Lemma \ref{sita}, we see that 
\begin{align*}
C(\Delta(r+2))-d-1+\left\lfloor \frac{17s-4}{6} \right\rfloor \geq \frac{2^{r+2}-2^{r-1}-1}{3} - (2^{r+2}-1)+\frac{17 \cdot 2^{r-1} -2}{3}=0 
\end{align*}
when $r$ is odd, and 
\begin{align*}
C(\Delta(r+2))-d-1+\left\lfloor \frac{17s-4}{6} \right\rfloor = \frac{2^{r+2}-1}{3} - (2^{r+2}-1)+\frac{17 \cdot 2^{r-1} -4}{3}=\frac{2^{r-1}-2}{3} \geq 0
\end{align*}
when $r$ is even with $r \geq 2$. Conjecture \ref{modify} for the case $r=0$ is trivially true. 

Hence, Conjecture \ref{modify} is true for $\Delta(r+2)$ (any lattice simplex of dimension $(4s-2)$ with degree $s$ which is not a lattice pyramid). 
Since $\Delta(r+2)$ is a lattice simplex having a maximal dimension with a given degree, 
the above statements can be one evidence implying that Conjecture \ref{modify} might be true in genera.